\documentclass[a4paper,12pt]{article}
\usepackage{amssymb,amsmath,graphicx,color, comment, xspace,amsthm}
\usepackage[colorlinks=true,citecolor=blue,urlcolor=blue,linkcolor=blue,bookmarksopen=true,unicode=true,pdffitwindow=true]{hyperref}
\usepackage[font=small,labelfont=bf,width=12.5cm]{caption}

\newtheorem{theorem}{Theorem}
\newtheorem{lemma}[theorem]{Lemma}

\newtheorem{proposition}[theorem]{Proposition}

\newcommand{\inst}[1]{$^{#1}$}
\newcommand{\set}[1]{\ensuremath{\left\{#1 \right\}}}
\newcommand{\total}{\ensuremath{16} }

\title{Improved bound on facial parity edge coloring}
\author
{
	Borut Lu\v{z}ar\inst{1}\thanks{Operation financed in part by the European Union, European Social Fund.},
	Riste \v{S}krekovski\inst{2}\thanks{Supported in part by ARRS Research Program P1-0297.}
}

\begin{document}
\maketitle

\begin{center}
{\footnotesize
	\inst{1} Institute of Mathematics, Physics, and Mechanics,\\
         Jadranska~19, 1111 Ljubljana, Slovenia\\
         E-Mail: borut.luzar@gmail.com\\
         \quad\\
	\inst{2} Department of Mathematics, University of Ljubljana\\
         Jadranska~21, 1111 Ljubljana, Slovenia\\
         E-Mail: skrekovski@gmail.com\\         
} \end{center}
\quad\\
\abstract{A \textit{facial parity edge coloring} of a $2$-edge connected plane graph is an edge coloring where 
			no two consecutive edges of a facial walk of any face receive the same color. Additionally, for every face
			$f$ and every color $c$ either no edge or an odd number of edges incident to $f$ are colored by $c$.			
			Czap, Jendrol', Kardo\v{s} and Sot\'{a}k~\cite{CzaJenKarSot12} showed that every $2$-edge connected plane graph
			admits a facial parity edge coloring with at most $20$ colors. We improve this bound to $\total$ colors.}
			
  
\section{Introduction}

A \textit{facial parity edge coloring} or an \textit{FPE-coloring} of a $2$-edge connected plane graph is an edge coloring where 
no two consecutive edges of a facial walk of any face receive the same color. Additionally, for every face
$f$ and every color $c$ either no edge or an odd number of edges incident to $f$ are colored by $c$.
The minimum number of colors such that an FPE-coloring of a graph $G$ exists is called the \textit{facial parity chromatic index}, 
$\chi_{fp}(G)$. 

In~\cite{CzaJenKar11}, Czap, Jendrol', and Kardo\v{s} defined the FPE-coloring of plane graphs and proved that $92$ colors suffice
to color any $2$-edge connected plane graph. Recently, Czap et al.~\cite{CzaJenKarSot12} improved the upper bound.
\begin{theorem}[Czap, Jendrol', Kardo\v{s}, and Sot\'{a}k]
	\label{thm:old}
	Let $G$ be a $2$-edge connected plane graph. Then,
	$$
		\chi_{fp}(G) \le 20\,.
	$$	
\end{theorem}
They also showed that if $G$ is $3$-edge connected, $12$ colors suffice, and in case when $G$ is a $4$-edge connected plane graph,
it has an FPE-coloring with at most $9$ colors.
In addition, Czap~\cite{Cza12} showed that facial parity chromatic index of outerplanar graphs is at most $15$
and presented the example, two pentagons with one vertex identified, which requires 10 colors for an FPE-coloring (see Fig.~\ref{fig:minimal}).
\begin{figure}[ht]
	$$
		\includegraphics{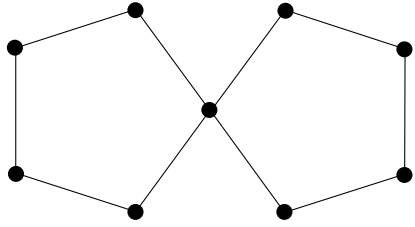}
	$$
	\caption{The graph that needs at least $10$ colors for an FPE-coloring.}
	\label{fig:minimal}
\end{figure}
In this paper we improve the upper bound for facial parity chromatic index of any $2$-edge connected plane graph.
\begin{theorem}
	\label{thm:main}
	Let $G$ be a $2$-edge connected plane graph. Then,
		$$
			\chi_{fp}(G) \le \total\,.
	$$
\end{theorem} 

\section{Proof of Theorem~\ref{thm:main}}

Before we prove Theorem~\ref{thm:main}, we present some definitions and the notation. 
A \textit{dual} $G^* = (V^*,E^*,F^*)$ of a plane graph $G = (V,E,F)$ 
is the graph with $V^* = V$ and two vertices $f$, $f'$ of $G^*$ are connected by an edge $e^*$ if the faces $f$ and $f'$ are incident
to an edge $e$ in $G$. We say that the edges $e^*$ and $e$ are \textit{associated}. An $H$-block of a graph $G$ is a block in $G$ isomorphic
to the graph $H$. Let $V$ be some subset of vertices of a graph $G$. As usual, $G[V]$ is a subgraph of $G$ induced by the vertices of $V$.

The \textit{facial walk} of a face $f$ is the shortest
closed walk containing all the edges from the boundary of $f$. Two edges are \textit{face-adjacent} if they are consecutive on some
facial walk. An edge coloring is \textit{facially proper} if the face-adjacent edges receive distinct colors. It is easy to see that
a facially proper edge coloring is equivalent to a vertex coloring of the medial graph of $G$, which is planar and therefore $4$-colorable
by the Four Color Theorem. 
We say that an edge coloring of $G$ is \textit{facially-odd} if it is facially proper and for any pair of adjacent faces $f$ and $f'$
holds that the number of common edges colored by the same color $c$ is odd or zero, for every color $c$.

\paragraph{}
In~\cite{CzaJenKarSot12}, the authors proved that for any connected plane graph there is a facially-odd edge coloring with
at most $5$ colors. The bound is best possible if the graph contains blocks isomorphic to $C_5$, a cycle on $5$ vertices. 
We say that an edge coloring is \textit{quasi-facially-odd} if it is facially-odd on all the edges, except for the edges
of $C_5$-blocks. The edges of every $C_5$-block are colored by $4$ colors such that the color that appears twice is assigned to 
the edges that are not face-adjacent.
\begin{lemma}
	\label{lem:facodd}	
	Every connected plane graph admits a quasi-facially-odd edge coloring with at most $4$ colors.
\end{lemma}
\begin{proof}	
	Suppose, for a contradiction, that $G$ is a minimal counterexample to the lemma, i.e., a connected plane graph with the minimum 
	number of edges that does not admit a quasi-facially-odd edge coloring with at most $4$ colors. 
	
	First, we show that $G$ is 
	$2$-connected. Suppose, to the contrary, that $v$ is a cut-vertex in $G$. Let $C_1$ be one of the components of $G - v$ and $C_2$
	the subgraph comprised of the remaining components.
	By the minimality of $G$, there exist quasi-facially-odd edge colorings $\varphi_1$, $\varphi_2$ with at most $4$ colors of the 
	graphs $C_1' = G[V(C_1) \cup \set{v}]$ and $C_2' = G[V(C_2) \cup \set{v}]$. Let $(e_1,e_1')$ and $(e_2,e_2')$ be the two pairs of face-adjacent edges in $G$, 
	where $e_1, e_2 \in C_1'$ and $e_1',e_2' \in C_2'$, see Figure~\ref{fig:bags}. 
	\begin{figure}[ht]
		$$
			\includegraphics{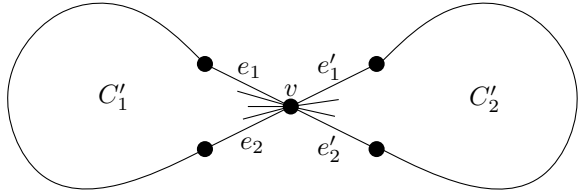}
		$$
		\caption{The face-adjacent edges of $C_1'$ and $C_2'$.}
		\label{fig:bags}
	\end{figure}
	Notice that the edge $e_1$ may be equal to $e_2$, and $e_1'$ may be equal to $e_2'$. 
	By permuting the colors of $\varphi_2$ such that $\varphi_1(e_1) \neq \varphi_2(e_1')$ and $\varphi_1(e_2) \neq \varphi_2(e_2')$
	we obtain a quasi-facially-odd edge coloring of $G$. Obviously, such a permutation exists. Thus, $G$ is $2$-connected. In particular, $G$
	does not contain a $C_5$-block, unless $G$ is isomorphic to $C_5$.
		
	If any pair of faces in $G$ has at most one edge in common, every facially proper edge coloring is also quasi-facially-odd. Thus, we may 
	assume that there is at least one pair of faces in $G$ that has at least two edges in common. Let $f$ and $f'$ be two faces with $k$ 
	common edges, $e_1,e_2,\dots,e_k$, $k \ge 2$, appearing in the given order on the facial walk of $f$. 
	First, we consider the case when two consecutive edges, $e_i$ and $e_{i+1}$, form a nontrivial 
	edge-cut in $G$, i.e., each of the two components, $D_1$ and $D_2$, of $G - \set{e_i, e_{i+1}}$ contains at least two vertices. 
	Let $H_1 = G / D_2$, i.e., $H$ is constructed from $G$ by contracting $D_2$ into a vertex, and similarly let $H_2 = G / D_1$. 
	
	In order to obtain a quasi-facially-odd edge coloring of $G$, we consider two cases. Suppose first that one of the subgraphs $H_1$ and $H_2$, 
	say $H_1$, is isomorphic to $C_5$ (see Figure~\ref{fig:special}). Observe that in this case $H_2$ is not isomorphic to $C_5$, by the choice of
	$e_i$ and $e_{i+1}$.
	\begin{figure}[ht]
		$$
			\includegraphics{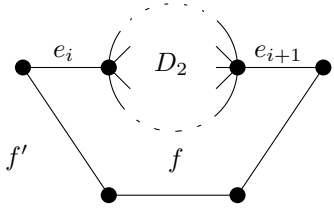}
		$$
		\caption{An example of the graph $G$ where $H_1$ is a $5$-cycle.}
		\label{fig:special}
	\end{figure}	
	By the minimality of $G$, there exists a quasi-facially-odd edge coloring $\rho$ of $H_2 / e_{i+1}$. We show that $\rho$ can
	be extended to $G$. Let $e_i$, $c_1$, $c_2$, $c_3$, and $e_{i+1}$ be the consecutive edges of $H_1$. We assign the color $\rho(e_i)$
	to $e_{i+1}$ and $c_3$, and color $c_2$ and $c_4$ with the two of the three remaining colors. Hence, extending $\rho$ to $G$.
	
	Suppose now that neither $H_1$ nor $H_2$ is isomorphic to $C_5$. By the minimality of $G$, there are quasi-facially-odd edge colorings 
	$\rho_1$ and $\rho_2$ with at most $4$ colors of $H_1$ and $H_2$, respectively. By permuting the colors of the edges of $H_2$ such that
	$\rho_1(e_i) = \rho_2(e_i)$ and $\rho_1(e_{i+1}) = \rho_2(e_{i+1})$, we obtain a quasi-facially-odd edge coloring of $G$. 
	Such a permutation always exists, since the edges $e_i$ and $e_{i+1}$ receive distinct colors in both graphs, $H_1$ and $H_2$, due to the face-adjacency.
	
	Hence, we may assume that no two consecutive common edges of $f$ and $f'$ form a nontrivial edge-cut. That means that the edges $e_1,e_2,\dots,e_k$ form a path
	or possibly a cycle if $f$ and $f'$ are the only faces in $G$, since $G$ is $2$-connected. We consider the subcases regarding the size of $k$.	
	\begin{itemize}
		\item{} $k = 2$.\quad By the minimality, there is a quasi-facially-odd edge coloring of $G / e_2$ with $4$ colors. Notice that
				at most three colors are forbidden for the only noncolored edge $e_2$, the colors of the face-adjacent edges. Hence we can color it.
		\item{} $k = 3$.\quad By the minimality, there is a quasi-facially-odd edge coloring of $G / e_2$ with $4$ colors. Only the colors
				of $e_1$ and $e_3$ are forbidden for $e_2$. Hence, we color $e_2$ with one of the two remaining colors.
		\item{} $k = 4$.\quad Similarly as in the previous case, we color $G / \set{e_2,e_3}$ with $4$ colors. The colors
				of $e_1$ and $e_4$ are forbidden for $e_2$ and $e_3$. We color them with the two remaining colors.
		\item{} $k = 5$.\quad Since $G$ is $2$-connected, the edges $e_1,e_2,\dots,e_5$ do not form a $C_5$-block in $G$, unless $G = C_5$.
				In case when $G = C_5$, we color the edge $e_1$ and $e_3$ by color $1$ and the remaining three edges by the remaining three colors.
				Otherwise, we color $G / \set{e_2,\dots,e_5}$ with $4$ colors. Then, we use the color of $e_1$ to color $e_3$ and $e_5$, 
				since $e_5$ is not face-adjacent with $e_1$, and two of the three remaining colors to color $e_2$ and $e_4$. 				
		\item{} $k \ge 6$.\quad If $G$ is not a $9$-cycle, we color $G / \set{e_2,\dots,e_5}$ with $4$ colors. Then, we use
				the color of $e_1$ to color $e_3$ and $e_5$ and the color of $e_6$ to color $e_2$ and $e_4$. Otherwise, $G = C_9$ and
				we color the nine edges by the colors $1$, $2$, and $3$ alternatingly.
	\end{itemize}
	It is easy to see that every color appears odd times or does not appear at all on the edges $e_1,e_2,\dots, e_k$. Hence, $G$ admits
	a quasi-facially-odd edge coloring with at most $4$ colors, a contradiction.
\end{proof} 
As a corollary of Lemma~\ref{lem:facodd}, we infer that if the plane graph is $2$-connected and not isomorphic to $C_5$, 
it admits a facially-odd edge coloring with at most $4$ colors.

\paragraph{}
An \textit{odd graph} is a graph where all the vertices have odd or zero degree. An \textit{odd edge coloring} of a graph $G$ 
is a coloring (not necessarily proper) of the edges of $G$ such that every color class
induces an odd subgraph. The minimum number of colors for which an odd edge coloring of $G$ exists is called 
the \textit{odd chromatic index} of $G$, $\chi_o'(G)$. Obviously, every odd graph admits an odd edge coloring by one color.
In 1991, Pyber~\cite{Pyb91} proved the following theorem.
\begin{theorem}[Pyber]
	\label{thm:odd}
	Let $G$ be a simple graph. Then,
	$$
		\chi_o'(G) \le 4\,.
	$$
\end{theorem}
The tight bound is realized by $W_4$, the wheel of $4$ spokes (the left graph in Figure~\ref{fig:oddmult}). For an odd edge coloring
of a multigraph we may need more than $4$ colors, e.g., the right graph in Figure~\ref{fig:oddmult} needs $6$ colors.
\begin{figure}[ht]
	$$
		\includegraphics{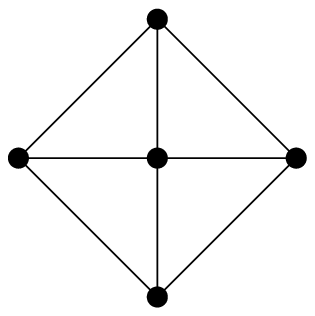} \quad\quad\quad\quad\quad \includegraphics{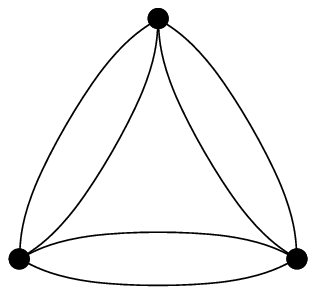}
	$$
	\caption{A simple graph with the odd chromatic index equal to $4$, and a multigraph with the odd chromatic index equal to $6$.}
	\label{fig:oddmult}
\end{figure}
However, if the given multigraph has some special properties, $4$ subgraphs suffice. The following simple result is mentioned in~\cite{Pyb91}
but we prove it here for the sake of completeness.
\begin{lemma}
	\label{lem:stars}
	Let $F$ be a forest. Then,
	$$
		\chi_o'(F) \le 2\,.
	$$
\end{lemma}
\begin{proof}
	It suffices to show that every tree $T$ has the odd chromatic index at most $2$. We will construct an odd edge coloring of $T$ with
	at most two colors. Let $v$ be an arbitrary vertex of $T$ with degree $d \ge 1$ (otherwise $T$ has no edges and so $\chi_o'(T) = 0$). If $d$ is odd, we color 
	all the edges incident to $v$ by $1$, otherwise we color $d-1$ incident edges by $1$ and the remaining with $2$. 
	
	Now, we continue by coloring the edges incident to the neighbors of $v$. Each neighbor $u$ already has one incident edge colored by some color, say $1$,
	thus we color all the other incident edges of $u$ with $1$, unless $d(u)$ is even. In the latter case we color them by color $2$. We repeat the same
	procedure on the vertices that have only one incident edge colored an eventually we obtain an odd edge coloring of $T$ with at most two colors.	
\end{proof}

A \textit{$k$-bridge} in a graph $G$ is a collection 
of $k \ge 1$ edges between two vertices whose deletion results in an increased number of components in $G$.
\begin{lemma}
	\label{prop:odd}
	Let $G$ be a loopless multigraph with a $k$-bridge. Then, 
	$$
		\chi_o'(G) \le 4\,.
	$$
\end{lemma}

\begin{proof}	
	Suppose, for a contradiction, that $G$ is a loopless multigraph with a $k$-bridge having the minimum number of edges 
	with $\chi_o'(G) \ge 5$. It immediately follows that $G$ is connected.
	
	We construct an odd edge coloring of $G$ with at most $4$ colors in the following way. Let $u$ and $v$ be the end-vertices of the $k$-bridge 
	in $G$ and let $G_u$ and $G_v$ be the two components of $G$ with the edges of the $k$-bridge removed. Note that one or both components may
	consist of only one vertex $u$ or $v$, respectively. We will separately	color the graphs $G_u' = G[V(G_u) \cup \set{v}]$ and $G_v' = G[V(G_v) \cup \set{u}]$ by
	at most $4$ colors and then show that the obtained colorings induce an odd edge coloring of $G$ with at most $4$ colors.
	
	In the sequel, we color $G_u'$ and the analogous procedure may be used to color $G_v'$. Let $T$ be a spanning tree of $G_u'$. 
	Note that the vertex $v$ is a leaf in $T$. Consider two cases regarding the number of vertices $n$ of $G_u'$. If $n$ is even, then 
	there is also an even number of vertices of even degree $v_1,v_2,\dots,v_{2\ell}$ in $G_u'$. Let $P_i$ be a path between $v_{2i-1}$ and 
	$v_{2i}$ in $T$ for $i \in \set{1,2,\dots,\ell}$. The disjoint union of all paths $P_i$ is a forest 
	$F = P_1 \oplus P_2 \oplus \dots \oplus P_\ell$ contained in $T$. By Lemma~\ref{lem:stars},
	every forest admits an odd edge coloring by at most $2$ colors, say $1$ and $2$. Observe also that the graph $G_u' - F$ is an odd graph, 
	which we can color by color $3$. Thus, we infer that $G_u'$ admits an odd edge coloring by at most $3$ colors.
	
	Assume now that $n$ is odd. The graph $G_u' - v$ is a connected graph with an even number of vertices, hence $T - v$ contains a forest $F$ such that 
	the graph $G_u' - v - F$ is an odd graph, which we color by color $3$. If the degree of $v$ is even in $G_u'$, we add the edge of $T$ incident to $v$ 
	to the forest $F$. As in the previous case, we can color $F$ by at most $2$ colors, $1$ and $2$. It remains to consider the edges incident to $v$ 
	that are not colored yet. 
	Recall that all the edges incident to $v$ in $G_u'$ are incident to $u$. The number of such edges is odd, so we color them by color $4$,
	obtaining an odd edge coloring of $G_u'$ with at most $4$ colors. 
	
	Notice that in both graphs, $G_u'$ and $G_v'$, we colored the $k$ edges between $u$ and $v$ similarly. If $k$ is odd, all the edges admit the same color,
	and in case when $k$ is even, one edge is colored by one color and all the rest by some other color. Hence, there exists a permutation of the colors of $G_v'$
	such that the edges between $u$ and $v$ receive the same colors in both graphs. Notice that both colorings together induce an odd edge coloring of $G$ 
	by at most $4$ colors which establishes the lemma.
\end{proof}
We will make use of the following observation presented in~\cite{CzaJenKarSot12}.
\begin{proposition}
	\label{prop:odd2}
	Let $G$ be a loopless multigraph with an odd number of edges between any pair of adjacent vertices. Then, 
	$$
		\chi_o'(G) \le 4\,.
	$$
\end{proposition}

Now, we prove Theorem~\ref{thm:main} in a similar manner as in~\cite{CzaJenKarSot12}.
\begin{proof}[Proof of Theorem~\ref{thm:main}.]
	Let $G = (V,E,F)$ be a $2$-edge connected plane graph. We will find an FPE-coloring of $G$ with $\total$ colors in two phases. 
	First, let $\varphi$ be a quasi-facially-odd edge coloring of $G$ with $4$ colors, which exists by Lemma~\ref{lem:facodd}. 
	In the second phase, we granulate the coloring $\varphi$ in such a way that every color appears odd or zero times on any facial walk in $G$. 
	
	Let $C$ be the set of $4$ colors used in $\varphi$. For every color $c \in C$ let $G_c^*$ be the graph with the vertex set 
	$V(G_c^*) = V(G^*)$ and the edge set $E(G_c^*)$ being a subset of the edges $E(G^*)$, where $G^*$ is the dual of $G$.
	In particular, the edge $e^*$ of $G^*$ is in $G_c^*$ if the color $c$ was assigned to the edge $e$ associated to $e^*$.
	Since $G$ is bridgeless, its dual $G^*$ is loopless and so is $G_c^*$, for any color $c$. Moreover, any two adjacent vertices of $G_c^*$ are
	connected by an odd number of edges, except for the pairs of vertices representing the faces incident to $C_5$-blocks in $G$, which 
	are connected with two edges. Notice that the two edges form a $2$-bridge in $G_c^*$, which separates the 
	component in the interior of a $C_5$-block from the component in the exterior. 
	
	By Proposition~\ref{prop:odd2} and Lemma~\ref{prop:odd}, there exists an odd edge coloring $\omega$ of the edges of $G_c^*$ by at most $4$ colors. 
	Now, we replace the color $c$ on the edge $e$ with the color $(c, i)$, where $i = \omega(e^*)$. The odd or zero degree 
	of every vertex in $G_c^*$ assures that every color $(c,i)$ appears odd number of times or not at all on the boundary of every face of $G$. Hence, 
	after the procedure described above is performed on all $4$ colors of $C$, we obtain an FPE-coloring of $G$ with at most $\total$ colors.
\end{proof}

\section{Conclusion}

Notice that our bound holds for graphs embedded on other surfaces if no face is incident to itself. Such a case
results in a multigraph with loops when defining the dual $G_c^*$ restricted to one color $c$. Whenever there is a component with
one vertex incident to several loops in a graph, it cannot be partitioned into edge-disjoint subgraphs. 

In the introduction we mention that there exist graphs that need at least $10$ colors for an FPE-coloring, hence there remains a gap
of $6$ colors between the upper and lower bound of the facial parity chromatic index of $2$-edge connected plane graphs. 
In the proof we use a combination
of facially-odd edge coloring and partition the edges of every color class into odd subgraphs using $4$ colors for each color.
However, the fourth color in the odd edge coloring appears only on the edges incident to one vertex. Therefore, we believe that the argument 
could be modified in such a way that the fourth color could be omitted if the first phase of the coloring is carefully performed.

\bibliographystyle{acm}
{\small
	\bibliography{MainBase}

\begin{thebibliography}{1}

\bibitem{Cza12}
{\sc Czap, J.}
\newblock Facial parity edge coloring of outerplane graphs.
\newblock {\em Ars Math. Contemp. 5\/} (2012), 285--289.

\bibitem{CzaJenKar11}
{\sc Czap, J., Jendrol', S., and Kardo\v{s}, F.}
\newblock Facial parity edge colouring.
\newblock {\em Ars Math. Contemp. 4\/} (2011), 255--269.

\bibitem{CzaJenKarSot12}
{\sc Czap, J., Jendrol', S., Kardo\v{s}, F., and Sot\'{a}k, R.}
\newblock Facial parity edge colouring of plane pseudographs.
\newblock {\em Discrete Math. 312\/} (2012), 2735--2740.

\bibitem{Pyb91}
{\sc Pyber, L.}
\newblock Covering the edges of a graph by ...
\newblock {\em Graphs and Numbers, Colloquia Mathematica Societatis J\'{a}nos
  Bolyai 60\/} (1991), 583--610.

\end{thebibliography}
}

\end{document}